\documentclass[10pt]{amsart}
\usepackage[utf8]{inputenc}
\usepackage{amsthm, caption}
\usepackage{amsmath, amssymb, tikz}

\usepackage{enumitem}
\usepackage[all]{xy}

\theoremstyle{plain}
\newtheorem{theorem}{Theorem}[section]

\newtheorem{definition}[theorem]{Definition}
\newtheorem{lemma}[theorem]{Lemma}
\newtheorem{proposition}[theorem]{Proposition}
\newtheorem{example}[theorem]{Example}
\newtheorem{remark}[theorem]{Remark}
\newtheorem{fact}[theorem]{Fact}
\newtheorem*{theorem*}{Theorem}

\theoremstyle{remark}

\DeclareMathOperator{\Aut}{Aut}
\DeclareMathOperator{\Th}{Th}
\DeclareMathOperator{\Part}{Part}
\DeclareMathOperator{\Age}{Age}

\DeclareMathOperator{\dom}{dom}
\DeclareMathOperator{\range}{range}

\DeclareMathOperator{\acl}{acl}

\usepackage{hyperref}
\usepackage{color, colortbl}
\hypersetup{pdfborder={0 0 0}}

\usepackage[a4paper, total={6in, 8in}]{geometry}

\renewcommand{\restriction}{\mathord{\upharpoonright}}

\binoppenalty=\maxdimen
\relpenalty=\maxdimen

\title{Bowtie-free graphs and generic automorphisms}
\author{Daoud Siniora}
\address{Department of Mathematics and Actuarial Science, The American University in Cairo}
\email{daoud.siniora@aucegypt.edu}

\begin{document}

\maketitle

\begin{abstract}       
We show that the countable universal $\omega$-categorical bowtie-free graph admits generic automorphisms. Moreover, we show that this graph is not finitely homogenisable.
\end{abstract}

\section{Introduction}

In this paper we examine the automorphism group of the universal $\omega$-categorical bowtie-free graph. This graph appears in the study of the question of existence of countably infinite universal graphs which forbid finitely many finite graphs as subgraphs, rather than induced subgraphs. The first examples of such universal graphs include the random graph and the universal homogeneous $K_n$-free graph for $n\geq 3$. We focus here on the case of the bowtie-free universal graph, where a bowtie ($\bowtie$) is the graph consisting of two triangles glued at one common vertex. A bowtie-free universal graph was first proved to exist by Komj\'{a}th \cite{komjath}, a result which was not attainable via the Fra\"{i}ss\'{e} amalgamation technique at the time. Such an obstacle provided the motivation behind the combinatorial theory developed by Cherlin, Shelah, and Shi \cite{cherlinshelahshi} which established the existence of an $\omega$-categorical universal bowtie-free graph (and other universal graphs) using the algebraic closure operator. We will denote this universal bowtie-free graph by $\mathcal{U}_{\bowtie}$.  
The Cherlin-Shelah-Shi theory and the uniqueness of $\mathcal{U}_{\bowtie}$ is discussed in Section \ref{CSSsection}. Hubi\v{c}ka and Ne\v{s}et\v{r}il \cite{hubickanesetril} were also interested in the universal bowtie-free graph as it serves as an interesting example which lies in the intersection of Ramsey theory and model theory. 

Let $M$ be a countably infinite first-order structure. We say that $M$ is \textit{homogeneous} if every isomorphism between finitely generated substructures
of $M$ extends to an automorphism of $M$. The automorphism group $\Aut(M)$ of $M$ is endowed
with the pointwise convergence topology which makes it a Polish group.  With respect to this topology, the structure $M$ is said to have \textit{generic automorphisms} if $\Aut(M)$ contains a comeagre conjugacy class — see Truss \cite{truss}. 

We motivate the significance of generic automorphisms by mentioning some of their group-theoretic consequences. Suppose that $G$ is a Polish group with a comeagre conjugacy class C. Then, we have that $G = C^2 = \{gh \mid g, h \in C\}$, and every element of $G$ is a commutator. Thus, $G = G'$ where $G'$ is the commutator subgroup. Moreover, if $G$ is uncountable, then $G$ has no proper normal subgroup of
countable index. See Macpherson \cite[Proposition 4.2.12]{macphersonsurvey}. Furthermore, we have the
following.

\begin{theorem*}[Macpherson-Thomas \cite{macphersonthomas}] Suppose that $G$ is a Polish group with a comeagre conjugacy class. Then $G$ is not a non-trivial free product with amalgamation.
\end{theorem*}

In Section \ref{bowtieFAPsection} we extend an amalgamation lemma in \cite{hubickanesetril} regarding a cofinal subclass
of the class of all finite bowtie-free graphs. Consequently, via a variation of Fraïssé’s
amalgamation technique, we obtain a universal bowtie-free graph isomorphic to the Cherlin-Shelah-Shi universal bowtie-free graph $\mathcal{U}_{\bowtie}$. Here is our main result (see Definition \ref{finitehomog} for the definition of finitely homogenisable). 

\begin{theorem*}
The universal $\omega$-categorical bowtie-free graph admits generic automorphisms. Moreover, it is not finitely homogenisable.
\end{theorem*}


\section{Preliminaries} \label{CSSsection}

The technique we use to establish the existence of generic automorphisms for the universal bowtie-free graph is the \textit{Kechris-Rosendal characterisation} \cite[Theorem 3.4]{kechrisrosendal}. Suppose that $A$ is a first-order structure. A \textit{partial automorphism} of $A$ is an isomorphism
$p:U \to V$ where $U$ and $V$ are substructures of $A$. We denote by $\Part(A)$ the set of all
partial automorphisms of $A$. Suppose that $\mathcal{C}$ is a Fraïssé class of finite $\mathcal{L}$-structures over some first-order language $\mathcal{L}$. An $n$-system over $\mathcal{C}$ is a tuple $\langle A, p_1, \ldots , p_n\rangle$ where $A$ is in $\mathcal{C}$ and each
$p_i \in \Part(A)$. Let $S=\langle A,p \rangle$ and $S'=\langle B,f \rangle$ be two 1-systems over $\mathcal{C}$. An embedding from $S$ to $S'$ is an $\mathcal{L}$-embedding $\phi:A\to B$ such that $\phi(\dom(p))\subseteq \dom(f)$, and $\phi(\range(p))\subseteq \range(f)$, and $\phi \circ p \subseteq f \circ \phi$. We say that the class of all $1$-systems over $\mathcal{C}$ has the \textit{joint embedding property} if for every two $1$-systems $S_1$ and $S_2$, there is a $1$-system $R$ which embeds both $S_1$ and $S_2$. Furthermore, the class of all $1$-systems over $\mathcal{C}$ has the \textit{weak amalgamation property} if for every $1$-system $S$ there exists some $1$-system $\hat{S}$ and embedding $\iota:S\to \hat{S}$ such that whenever $S_1$ and $S_2$ are $1$-systems with embeddings $\alpha_1:\hat{S}\to S_1$ and $\alpha_2:\hat{S}\to S_2$, there is a $1$-system $R$ with embeddings $\beta_1:S_1\to R$ and $\beta_2:S_2\to R$ such that $\beta_1 \alpha_1 \iota = \beta_2 \alpha_2 \iota$. 

\begin{figure}[hp]
\begin{displaymath}
    \xymatrix{& & & S_1 \ar@{-->}[rd]^{\beta_1} \\    
           S\ar@/^1pc/[rrru]^{\alpha_1 \iota}\ar@/_1pc/[rrrd]_{\alpha_2 \iota}\ar@{-->}[rr]^{\iota}& &  \hat{S}\ar[ur]^{\alpha_1} \ar[dr]_{\alpha_2} & & R \\   
              & & & S_2 \ar@{-->}[ru]_{\beta_2}}
\end{displaymath}
\centering
\caption{Weak amalgamation property}
\end{figure}
Now, suppose that $M$ is the Fraïssé limit of $\mathcal{C}$. Then, $M$ has generic automorphisms if and only if the class of
1-systems over $\mathcal{C}$ has the joint embedding property and the weak amalgamation property. See \cite[Theorem 3.4]{kechrisrosendal}.

Next, we present a model theoretic approach developed in Cherlin-Shelah-Shi \cite{cherlinshelahshi} to the problem of existence of a universal graph with forbidden subgraphs. Let $\mathcal{F}$ be a family of finite graphs, viewed as `forbidden' graphs. A graph $G$ is called \textit{$\mathcal{F}$-free} if no graph in $\mathcal{F}$ is isomorphic to a subgraph of $G$ (not necessarily induced subgraph). That is, $G$ is $\mathcal{F}$-free if there exists no injective homomorphism from an element in $\mathcal{F}$ into $G$. Denote by $\mathcal{C}_\mathcal{F}$ the class of all countable (finite and countably infinite) $\mathcal{F}$-free graphs. A graph $G\in \mathcal{C}_\mathcal{F}$ is \textit{universal} for $\mathcal{C}_\mathcal{F}$ if every graph in $\mathcal{C}_\mathcal{F}$ is isomorphic to an \textit{induced} subgraph of $G$. For some graphs $G$ and $H$, by $G\subseteq H$ we mean that $G$ is an induced subgraph of $H$. 

\begin{figure}[hp]
\begin{tikzpicture}
\tikzset{Bullet/.style={circle,draw,fill=black,scale=0.5}}
\node[Bullet](a) at(0,0.5){};
\node[Bullet] (b)at(1.5,0){};
\node[Bullet] (d)at(-1.5,0){};
\node[Bullet] (c)at(1.5,1){};
\node[Bullet] (e)at(-1.5,1){};
\draw[thick](a)--(b)--(c)--(a)--(d)--(e)--(a) {};
\end{tikzpicture}
\centering
\caption{The bowtie graph}
\end{figure}

We collect below some results on the existence of countable universal graphs. Towards this, we first describe a graph generalising the bowtie graph. Given a collection $K_{n_1}, K_{n_2}, \ldots, K_{n_k}$ of complete graphs, their \textit{bouquet} $K_{n_1} + K_{n_2} + \ldots + K_{n_k}$ is the graph formed by taking the free amalgam of the given complete graphs over one common vertex. The bouquet $K_3 + K_3$ is the bowtie graph.

\begin{example}\rm Existence results.
\begin{enumerate}[label=(\roman*), itemsep=1pt,topsep=-3pt]
\item (Rado \cite{rado}). The class $\mathcal{C}_\emptyset$ of all countable graphs has a universal element. 

\item (Komj\'{a}th \cite{komjath}). There is a countable universal bowtie-free graph.

\item (Cherlin-Shi \cite{cherlinshi}). Suppose that $\mathcal{F}$ is a finite set of cycles. Then there is a countable universal $\mathcal{F}$-free graph if and only if $\mathcal{F}=\{C_3, C_5, C_7, \ldots, C_{2k+1}\}$ for some $k\geq 1$.   

\item (Cherlin-Tallgren \cite{cherlintallgren}). Let $F=K_m+K_n$ be a bouquet where $m\leq n$. Then there is a countable universal $F$-free graph if and only if $1\leq m \leq 5$ and $(m,n)\neq (5,5)$.
\end{enumerate}
\end{example}

Recall that a graph is \textit{$2$-connected} if it is connected, and remains connected after deleting any of its vertices. 

\begin{example}\rm Non-existence results.
\begin{enumerate}[label=(\roman*), itemsep=1pt,topsep=-3pt]
\item (Komj\'{a}th \cite{komjath}). Let $m,n \geq 3$. If $F=m\cdot K_n$, the bouquet of $m$-many copies of $K_n$, then there is no $F$-free countable universal graph.      

\item (Cherlin-Komj\'{a}th \cite{cherlinkomjath}). There is no countable universal $C_n$-free graph for $n\geq 4$. Here $C_n$ is a cycle of length $n$. (See item (iii) in the previous example.)

\item (F\"{u}redi-Komj\'{a}th \cite{furedikomjath}). If $F$ is a finite, $2$-connected, but not complete graph, then there is no countable universal $F$-free graph.
\end{enumerate}
\end{example}

We work with the language of graphs $\mathcal{L}=\{E\}$ where $E$ is a binary relation symbol. Denote by $T_\mathcal{F}$ the theory of the class $\mathcal{C}_\mathcal{F}$, that is, the theory $T_\mathcal{F}$ is the set of all $\mathcal{L}$-sentences true in all members of $\mathcal{C}_\mathcal{F}$. Note that $T_\mathcal{F}$ is equivalent to a universal theory.   

\begin{definition}\rm \cite[Definition 2]{cherlinshelahshi}
\begin{enumerate}[label=(\roman*), noitemsep,topsep=-3pt]
\item Let $H$ be a graph, and $G\subseteq H$ an induced subgraph. We say that $G$ is \textit{existentially closed} in $H$ if for every existential sentence $\exists\bar{x}\phi(\bar{x})$ with parameters from $G$ we have that if $H\models \exists\bar{x}\phi(\bar{x})$, then $G\models \exists\bar{x}\phi(\bar{x})$.  

\item A graph $G \in \mathcal{C}_\mathcal{F}$ is called \textit{existentially closed} in $\mathcal{C}_\mathcal{F}$ if $G$ is existentially closed in every graph $H\in \mathcal{C}_\mathcal{F}$ containing $G$ as a subgraph. 

\item Denote by $\mathcal{E}_\mathcal{F}$ the class of all existentially closed graphs in $\mathcal{C}_\mathcal{F}$. And let $T_\mathcal{F}^{ec}$ be the theory of the class $\mathcal{E}_\mathcal{F}$. 
\end{enumerate}
\end{definition}

\begin{remark}\rm
A graph $G\subseteq H$ being existentially closed in $H$ is equivalent to the following condition: if $A\subseteq B$ are finite graphs such that $A\subseteq G$ and $B\subseteq H$, then there is an embedding $f: B \to G$ such that $f\restriction_A$ is the identity map.
\end{remark}

The notions above apply to other settings. For example, the existentially closed structures in the class of fields are the algebraically closed fields. The existentially closed structures in the class of ordered fields are the real closed fields. Dense linear orders without endpoints are existentially closed in the class of linear orders. Existentially closed structures appear in model theory in Abraham Robinson's work on model complete theories---see \cite[Chapter 3]{marcjatoffalori}, \cite[Section 3.5]{changkeisler}, and \cite{hirschfeldwheeler}. A first-order theory $T$ is said to be \textit{model complete} if whenever $M, N\models T$ and $M\subseteq N$, then $M \preceq N$. Robinson's Test \cite[Theorem 3.2.1]{marcjatoffalori} states that the following are equivalent for an $\mathcal{L}$-theory $T$:

\begin{enumerate}[label=(\roman*), noitemsep,topsep=-3pt]
\item $T$ is model complete.

\item Whenever $M, N \models T$ with $M \subseteq N$, then $M$ is existentially closed in $N$.

\item Every $\mathcal{L}$-formula is equivalent to an existential formula modulo $T$.

\item Every $\mathcal{L}$-formula is equivalent to a universal formula modulo $T$.
\end{enumerate}

Suppose that $\mathcal{K}$ is an elementary class of $\mathcal{L}$-structures which is closed under unions of chains. Then every element $M\in \mathcal{K}$ can be extended to an element $\bar{M}\in \mathcal{K}$ which is existentially closed in $\mathcal{K}$ \cite[Lemma 3.5.7]{changkeisler}. Let $\mathcal{E}(\mathcal{K})$ be the subclass of all existentially closed structures in $\mathcal{K}$. Then $\mathcal{E}(\mathcal{K})$ may not be an elementary class. For instance, Eklof and Sabbagh proved that the class of existentially closed groups is not elementary \cite[Theorem 3.5.7]{marcjatoffalori}.    

\begin{proposition}{\rm \cite[Proposition 3.5.15]{changkeisler}} Let $\mathcal{K}$ be an elementary class of $\mathcal{L}$-structures closed under unions of chains. Let $T:=\Th(\mathcal{K})$ and $T^{ec}:=\Th(\mathcal{E}(\mathcal{K}))$. Then $T^{ec}$ is model complete if and only if $\mathcal{E}(\mathcal{K})$ is elementary.
\end{proposition}

We now get back to our setting of graphs. Cherlin, Shelah, and Shi proved the following which in view of the proposition above shows that $T_\mathcal{F}^{ec}$ is model complete when the family $\mathcal{F}$ is finite.

\begin{theorem}{\rm \cite[Theorem 1]{cherlinshelahshi}} Let $\mathcal{F}$ be a finite family of finite graphs. Then a countable graph $G \in \mathcal{E}_\mathcal{F}$ if and only if $G \models T_\mathcal{F}^{ec}$. Moreover, if every $F\in \mathcal{F}$ is connected, then $T_\mathcal{F}^{ec}$ is a complete theory.
\end{theorem}

\begin{example}\rm \cite[Example 4]{cherlinshelahshi} Let $\mathcal{F}=\{S_3\}$ where $S_3$ is a star of degree $3$, that is, a graph of $4$ vertices where one vertex is adjacent to the other three, and there are no more edges. Then $T_{\mathcal{F}}$ is the theory of graphs in which every vertex has degree at most $2$. And $T_\mathcal{F}^{ec}$ is the theory of graphs in which every vertex has degree $2$, and which contain infinitely many cycles $C_n$ for each $n\geq 3$. Let $\mathbb{Z}$ be the $2$-way infinite path, that is, vertices are the integers, and every $n$ is adjacent to $n+1$. Then a countable model of $T_\mathcal{F}^{ec}$ is characterised up to isomorphism by the number of its connected components isomorphic to $\mathbb{Z}$. Let $G_k\models T_\mathcal{F}^{ec}$ be the countable model with $k$-many components isomorphic to $\mathbb{Z}$. Then $\mathcal{E}_\mathcal{F}=\{G_k : k\in \omega+1\}$. Moreover $G_\omega \in \mathcal{C}_\mathcal{F}$ is a universal $S_3$-free graph. Remember that the members of $\mathcal{C}_\mathcal{F}$ and $\mathcal{E}_\mathcal{F}$ are countable.       
\end{example}

\begin{definition}\rm 
Suppose that $M$ is an $\mathcal{L}$-structure, and let $A\subseteq M$. The \textit{algebraic closure} $\acl_M(A)$ of $A$ in $M$ is the union of all finite $A$-definable subsets of $M$.   
\end{definition}

\begin{theorem}{\rm \cite[Theorem 3]{cherlinshelahshi}} Let $\mathcal{F}$ be a finite family of connected finite graphs. Then the following are equivalent.
\begin{enumerate}[label=(\roman*), itemsep=-1pt, topsep=-6pt]
\item The theory $T_\mathcal{F}^{ec}$ is $\omega$-categorical.

\item For any $M\models T_\mathcal{F}^{ec}$ and finite $A\subseteq M$, we have that $\acl_M(A)$ is finite.
\end{enumerate}
\end{theorem}

\begin{proposition}{\rm \cite[Proposition 1]{cherlinshelahshi}}\label{acl4} Let $G\in \mathcal{E}_{\bowtie}$ be a countable existentially closed bowtie-free graph, and let $A\subseteq G$ be finite. Then $|\acl_G(A)|\leq 4|A|$. 
\end{proposition}

As every graph $G\in \mathcal{C}_\mathcal{F}$ embeds in some graph $\bar{G}\in \mathcal{E}_\mathcal{F}$, we have that $\mathcal{C}_\mathcal{F}$ contains a universal element if and only if $\mathcal{E}_\mathcal{F}$ contains a universal element. Therefore, by the last two theorems and proposition above we have that $\mathcal{E}_{\bowtie}=\{G \text{ graph}: G\models T_{\bowtie}^{ec} \text{ and } |G|=\aleph_0\}$ contains exactly one element; an $\omega$-categorical existentially closed universal bowtie-free graph.


\section{Bowtie-free Graphs}\label{bowtieFAPsection}

Let $\mathcal{L}=\{E\}$ be the language of graphs. Recall that for graphs $G, H$, by $G\subseteq H$ we mean that $G$ is an induced subgraph of $H$.  A bowtie $(\bowtie)$ is the graph formed by freely amalgamating two triangles over one common vertex. And a graph is called bowtie-free if it has no subgraph isomorphic to the bowtie (not necessarily induced subgraph). We denote by $\mathcal{C}_{\bowtie}$ the class of all countable bowtie-free graphs, and denote by $\mathcal{C}^0_{\bowtie}$ the class of all finite bowtie-free graphs. Notice that a graph is bowtie-free if and only if it has no induced subgraph isomorphic to a graph $B$ where $\bowtie\ \subseteq B \subseteq K_5$, where $K_n$ denotes the complete graph on $n$ vertices.

Following Hubi\v{c}ka and Ne\v{s}et\v{r}il \cite{hubickanesetril}, a \textit{chimney} is the graph formed by taking the free amalgam of \textit{two} or more triangles over one common edge. Moreover, we expand the terminology as follows. We call the vertices of the common edge \textit{base vertices}, and the rest we call them \textit{tip vertices}. We also call the number of tip vertices the \textit{height} of the chimney. Any chimney contains exactly two base vertices, and at least two tip vertices.

In \cite{cherlinshelahshi}, an edge in a bowtie-free graph is called a \textit{special edge} if it lies in at least two triangles. For example, the base edge of a chimney is special, and every edge in $K_4$ is special. Motivated by special edges we present the following definition. 

\begin{definition}\rm 
A bowtie-free graph is called \textit{special} if every vertex is contained in a triangle that contains a special edge. Let $\mathcal{C}_{\bowtie}^{\ast}$ denote the class of all finite special bowtie-free graphs.  
\end{definition}

Observe that chimneys and $K_4$ are special bowtie-free graphs with the additional property that every edge is contained in some triangle.   

\begin{fact}[\cite{cherlinshelahshi}, \cite{hubickanesetril}]\rm Suppose that $G$ is a finite connected bowtie-free graph such that every edge is contained in some triangle. If $K_4\subseteq G$, then $G\cong K_4$. Otherwise, $G$ is a chimney or a triangle.
\end{fact}

In \cite[Definition 2.2]{hubickanesetril}, a bowtie-free graph is called \textit{good} if every vertex is contained either in a $K_4$ or in a chimney. One can check that the two notions, special and good, are equivalent.  

\begin{figure}[ht]
\centering
\begin{tikzpicture}
\tikzset{Bullet/.style={circle,draw,fill=black,scale=0.4}}
\node[Bullet](a) at(3,0){};
\node[Bullet] (d)at(0,0){};
\node[Bullet] (u)at(1.5,1){};
\node[Bullet] (s)at(1.5,2){};
\node[Bullet] (t)at(1.5,3){};
\node[Bullet] (r)at(1.5,4){};

\node[Bullet] (e)at(11,1){};
\node[Bullet] (f)at(13,1){};
\node[Bullet] (g)at(12,1.5){};
\node[Bullet] (h)at(12,2.3){};

\node[Bullet] (e1)at(4,2.5){};
\node[Bullet] (f1)at(5,4){};
\node[Bullet] (g1)at(4,3.5){};
\node[Bullet] (h1)at(3,4){};

\node[Bullet](a1) at(9,0.5){};
\node[Bullet] (d1)at(7,0.5){};
\node[Bullet] (u1)at(8,1.5){};
\node[Bullet] (s1)at(8,2.5){};
\node[Bullet] (t1)at(8,3.5){};

\draw[line width=0.4mm](a)--(d)--(u)--(a)--(s)--(d)--(t)--(a)--(r)--(d) {};
\draw[line width=0.4mm](e)--(f)--(g)--(h)--(e)--(g)--(f)--(h) {};
\draw[line width=0.4mm](e1)--(f1)--(g1)--(h1)--(e1)--(g1)--(f1)--(h1) {};
\draw[line width=0.4mm](a1)--(d1)--(u1)--(a1)--(s1)--(d1)--(t1)--(a1) {};
\draw(a1)--(e){};
\draw(u)edge[bend right=10](u1){};
\draw(u)edge[bend right=10](s1){};
\draw(u1)edge[bend right=10](h){};
\draw(e1)--(r)--(e1)--(t)--(e1)--(s)--(e1)--(s1)--(e1)--(t1){};
\draw(t1)--(h)--(s1){};
\end{tikzpicture}

\caption{A special bowtie-free graph. Thick edges lie in triangles, thin edges don't.} 
\end{figure}

\begin{fact}[\cite{hubickanesetril}] \rm 
Let $G$ be a special bowtie-free graph. By deleting all the edges of $G$ which do not lie in any triangle, we obtain a disjoint union of copies of $K_4$ and chimneys.
\end{fact}

Therefore, we get the following observation.

\textbf{Observation}. A finite graph $G$ is a special bowtie-free graph if and only if the vertex set of $G$ can be partitioned into disjoint induced subgraphs $H_1, H_2, \ldots, H_n$ where each $H_i$ is isomorphic to either a chimney or $K_4$, and moreover for every vertex $v\in H_i$ we have that the set of neighbours of $v$ not in $H_i$ forms an independent set (no two vertices are adjacent).

\begin{lemma}[\cite{hubickanesetril}]\label{specialcofinallemma}
The subclass $\mathcal{C}_{\bowtie}^{\ast}$ of special bowtie-free graphs is cofinal in the class $\mathcal{C}^0_{\bowtie}$ of finite bowtie-free graphs.   
\end{lemma}

\begin{proof}
We show that any finite bowtie-free graph is an induced subgraph of a special bowtie-free graph. Let $G$ be a finite bowtie-free graph. Suppose $v\in G$ is a vertex that is not contained in a triangle with a special edge. If $v$ belongs to a triangle (whose none of its edges are special), say $vxz$, then add a new vertex $u$ together with edges $uv$, $ux$, and $uz$, making $vxzu$ isomorphic to a $K_4$. Otherwise, $v$ is not contained in a triangle. In this case, add a new copy of $K_4$ and identify $v$ with one of its vertices. One can show that neither of these two actions will introduce a bowtie. Repeat this process until every vertex is contained in a triangle with a special edge.  
\end{proof}

Let $\mathcal{L}$ be a first-order language. We say that a class $\mathcal{C}$ of finite $\mathcal{L}$-structures has the \textit{amalgamation property} if whenever $A,B,C$ are in $\mathcal{C}$ and $f_1:A\to B$, $f_2:A\to C$ are embeddings, there is $D$ in $\mathcal{C}$ and embeddings $g_1:B\to D$, $g_2:C\to D$ such that $g_1\circ f_1=g_2\circ f_2$. See \cite[Section 2.1]{macphersonsurvey} for more details. 

The class $\mathcal{C}^0_{\bowtie}$ of finite bowtie-free graphs has the joint embedding property, however, $\mathcal{C}^0_{\bowtie}$ does not have the amalgamation property. To see this, consider the three graphs below. 

\begin{figure}[ht]
\begin{tikzpicture}
\tikzset{Bullet/.style={circle,draw,fill=black,scale=0.5}}
\node[Bullet,label=below:{$b$}](a) at(0,0){};
\node[Bullet, label=below:{$c$}] (b)at(1,0){};
\node[Bullet, label=below:{$a$}] (d)at(-1,0){};
\node[Bullet, label=left:{$u$}] (u)at(-1,1){};
\draw[thick](a)--(b)--(a)--(d)--(u)--(a) {};
\end{tikzpicture}
\hspace{2cm}
\begin{tikzpicture}
\tikzset{Bullet/.style={circle,draw,fill=black,scale=0.5}}
\node[Bullet,label=below:{$b$}](b) at(0,0){};
\node[Bullet, label=below:{$c$}] (c)at(1,0){};
\node[Bullet, label=below:{$a$}] (a)at(-1,0){};
\draw[thick](a)--(b)--(c) {};
\end{tikzpicture}
\hspace{2cm}
\begin{tikzpicture}
\tikzset{Bullet/.style={circle,draw,fill=black,scale=0.5}}
\node[Bullet,label=below:{$b$}](a) at(0,0){};
\node[Bullet, label=below:{$c$}] (b)at(1,0){};
\node[Bullet, label=below:{$a$}] (d)at(-1,0){};
\node[Bullet, label=right:{$v$}] (v)at(1,1){};
\draw[thick](a)--(b)--(v)--(a)--(d) {};
\end{tikzpicture}\\
$B$ \hspace{4.2 cm} $A$ \hspace{4.2 cm} $C$
\caption{Any amalgamation of $B$ and $C$ over $A$ (with the identity embeddings) contains a bowtie as a subgraph.}
\end{figure}

\begin{definition}\rm
Let $\mathcal{L}$ be a relational language. 
\begin{itemize}
    \item We say that a class $\mathcal{C}$ of finite $\mathcal{L}$-structures has the \textit{free amalgamation property} if whenever $A,B,C$ are in $\mathcal{C}$ and $f_1:A\to B$, $f_2:A\to C$ are embeddings, there is $D$ in $\mathcal{C}$ and embeddings $g_1:B\to D$, $g_2:C\to D$ such that:
\begin{enumerate}
    \item $g_1\circ f_1=g_2\circ f_2$,
    \item $g_1(B)\cap g_2(C)=g_1\circ f_1(A)$,
    \item If $D\models R(\bar{d})$ for some tuple $\bar{d}\in D$ and relation symbol $R\in\mathcal{L}$, then $\bar{d}\in g_1(B)$ or $\bar{d}\in g_2(C)$.
\end{enumerate}

\item Let $A$, $B$, and $C$ be $\mathcal{L}$-structures such that $A$ is a common substructure of $B$ and $C$. The \textit{free amalgam} of $B$ and $C$ over $A$ is the $\mathcal{L}$-structure $D$ whose underlying set is the disjoint union of $B$ and $C$ over $A$, and for every relation symbol $R\in \mathcal{L}$ we define its interpretation in $D$ by setting $R^D=R^B\cup R^C$. 
\end{itemize}
\end{definition}

The following proposition is of a more general form than \cite[Lemma 3.1]{hubickanesetril} where special bowtie-free graphs are amalgamated over their induced subgraph on bases of chimneys and copies of $K_4$.

\begin{proposition}\label{FAP}
The class $\mathcal{C}_{\bowtie}^{\ast}$ of all finite special bowtie-free graphs has the free amalgamation property.
\end{proposition}

\begin{proof}
Suppose $A, B_1, B_2$ are finite special bowtie-free graphs such that $A\subseteq B_1$ and $A\subseteq B_2$. Let $C$ be the free amalgam of $B_1$ and $B_2$ over $A$. Thus, the vertex set of $C$ is the disjoint union of the vertices of $B_1$ and $B_2$ with the vertices of $A$ identified, and two vertices in $C$ are adjacent if and only if either they are both in $B_1$ and adjacent, or both are in $B_2$ and adjacent.  We will show that $C\in \mathcal{C}_{\bowtie}^{\ast}$. By free amalgamation, any triangle in $C$ either lives entirely in $B_1$ or entirely in $B_2$. For the sake of contradiction, suppose $C$ has a bowtie $W=\{a,b,c,u,v\}$ as a subgraph where $c$ is the vertex of degree four, and $abc$ and $cuv$ are the triangles. As $B_1$ and $B_2$ are bowtie-free, it must be that $W$ is neither contained entirely in $B_1$ nor in $B_2$. Moreover, the vertex $c$ must be in $A$. To see this suppose $c\notin A$, then say $c\in B_1\setminus A$, and as $W\not\subseteq B_1$, one vertex say $u\in B_2\setminus A$, but $c$ and $u$ are adjacent contradicting free amalgamation. Consequently, we may now assume without loss of generality that $a,b,c\in B_1$ with $a\in B_1\setminus A$, and $c,u,v\in B_2$ with $u\in B_2\setminus A$. 

By the hypothesis, $A$ is a special bowtie-free graph, so the vertex $c$ is either contained in a $K_4$ of $A$, or in a chimney of $A$. Supposing the former, then the triangle $abc$ together with any triangle in that $K_4 \subseteq A$ which contains $c$ but not $b$ will form a bowtie inside $B_1$, contradicting that $B_1$ is bowtie-free. So $c$ must be contained in a chimney $Y$ of $A$. There are five possibilities in this situation, based on whether $c$ is a tip or a base vertex of $Y$. All lead to a contradiction. 

Case 1: Suppose that $c$ is a tip vertex of $Y$, and $b\in Y$. Then $b$ must be a base vertex of $Y$ as it is adjacent to $c$, and so the triangle $abc$ with any triangle of $Y$ not containing $c$ will form a bowtie in $B_1$, a contradiction. 

Case 2: Suppose that $c$ is a tip vertex of $Y$, and $b\notin Y$. Then the triangle $abc$ with the triangle in $Y$ containing $c$ form a bowtie in $B_1$, a contradiction. 

Case 3: Suppose that $c$ is a base vertex of $Y$ and $b \notin Y$. Then the triangle $abc$ together with any triangle in $Y$ will form a bowtie in $B_1$, a contradiction.

Case 4: Suppose that $c$ is a base vertex of $Y$ and $b$ is a tip vertex of $Y$. Then the triangle $abc$ with another triangle of $Y$ not containing $b$ will form a bowtie in $B_1$, a contradiction.

Case 5:  Suppose that $b, c$ are the base vertices of $Y$. In this case, the triangle $cuv$ together with any triangle of $Y$ not containing the vertex $v$ will form a bowtie in $B_2$, a contradiction.

Hence, the free amalgam $C$ is bowtie-free. Now we show $C$ is special. Any vertex $x\in C$ is either in $B_1$ or $B_2$. Say $x\in B_1$. As $B_1$ is special, the vertex $x$ lies in a triangle $R$ which contains a special edge $e$. As $B_1$ is an induced subgraph of $C$, both $R$ and the other triangle sharing $e$ are also in $C$. Therefore, $C$ is a special bowtie-free graph.           
\end{proof}

We now know that the class $\mathcal{C}_{\bowtie}^{\ast}$ of all finite special bowtie-free graphs has the free amalgamation property. Moreover, it is closed under disjoint unions, and so it has the joint embedding property. However $\mathcal{C}_{\bowtie}^{\ast}$ is not closed under induced subgraphs, that is, it does not have the hereditary property. In this situation, we can apply a slight variation of Fra\"{i}ss\'{e}'s Theorem which does not require the class of finite structures in hand to have the hereditary property. More precisely, we apply Kueker-Laskowski \cite[Theorem 1.5]{kuekerlaskowski} to the class $\mathcal{C}_{\bowtie}^{\ast}$ and obtain the following.

\begin{theorem}\label{universalbowtiefreegraph}
There is a unique, up to isomorphism, graph $\mathcal{U}_{\bowtie}$ such that:
\begin{enumerate}[label=(\roman*), itemsep=-1pt, topsep=-6pt]
\item The graph $\mathcal{U}_{\bowtie}=\bigcup\limits_{i\in\omega}G_i$ where $G_i\in \mathcal{C}_{\bowtie}^{\ast}$ and $G_i\subseteq G_{i+1}$ for all $i\in\omega$.

\item Every $H\in \mathcal{C}_{\bowtie}^{\ast}$ embeds into $\mathcal{U}_{\bowtie}$.

\item Every finite isomorphism $f:G\to H$ where $G,H \in \mathcal{C}_{\bowtie}^{\ast}$ and $G,H \subseteq \mathcal{U}_{\bowtie}$ extends to an automorphism of $\ \mathcal{U}_{\bowtie}$.
\end{enumerate}
\end{theorem}

We know that $\mathcal{C}_{\bowtie}^{\ast}$ is cofinal in $\mathcal{C}^0_{\bowtie}$. Consequently, by  Kueker-Laskowski \cite[Lemma 2.4]{kuekerlaskowski}, $\mathcal{U}_{\bowtie}$ of Theorem \ref{universalbowtiefreegraph} above is an existentially closed model of the universal theory $T_{\bowtie}$, that is, $\mathcal{U}_{\bowtie}\in \mathcal{E}_{\bowtie}$. By Cherlin-Shelah-Shi \cite{cherlinshelahshi} the theory of existentially closed bowtie-free graphs is $\omega$-categorical. Therefore, the graph $\mathcal{U}_{\bowtie}$ is isomorphic to the $\omega$-categorical universal countable bowtie-free graph of Cherlin-Shelah-Shi introduced at the end of the previous section. 

We aim now to describe the algebraic closure of a finite induced subgraph of the universal bowtie-free graph $\mathcal{U}_{\bowtie}$. Recall that in \cite{cherlinshelahshi}, an edge in $\mathcal{U}_{\bowtie}$ is called a \textit{special edge} if it lies in two triangles of $\ \mathcal{U}_{\bowtie}$. It was shown in \cite[Proposition 1]{cherlinshelahshi} that: 
\begin{enumerate} [label=(\roman*), itemsep=0pt, topsep=0pt]
    \item Every triangle in $\mathcal{U}_{\bowtie}$ contains a special edge. 
    \item If a vertex $v\in \mathcal{U}_{\bowtie}$ lies in a triangle $R$, but not in a special edge of $R$, then $v$ lies in a unique triangle. 
    \item If a vertex $v\in \mathcal{U}_{\bowtie}$ lies in two special edges, then $v$ lies in some $Q\cong K_4$, and thus any triangle containing $v$ is contained in $Q$. 
\end{enumerate}
It was shown further that for a finite induced subgraph $A\subseteq \mathcal{U}_{\bowtie}$, we have that
\begin{equation} \label{acl}
\acl(A)=A \cup \bigcup\Big\{e \text{ is a special edge of } \mathcal{U}_{\bowtie} \mid e \text{ lies in a triangle } R \text{ with } R \cap A \neq \emptyset\Big\}. \tag{$\dagger$} 
\end{equation}
Observe that  \eqref{acl} implies Proposition \ref{acl4}. In \eqref{acl} and below, we think of an edge $e$ as the set of the two vertices forming the edge $e$.

As $\mathcal{U}_{\bowtie}$ is existentially closed, one can see that every vertex $v\in \mathcal{U}_{\bowtie}$ lies in some triangle. By $(\mathrm{i})$ and $(\mathrm{iii})$ every triangle $R$ in $\mathcal{U}_{\bowtie}$ either contains exactly one special edge or contains three special edges. In the former case, $(\mathrm{ii})$ implies that $R$ lies in a chimney. In the latter case, $R$ lies in some $K_4$. So to sum up, \textit{every vertex and every triangle in $\mathcal{U}_{\bowtie}$ lies in a chimney or a $K_4$}. Also note that in a chimney, there is only one special edge, namely the edge between the two base vertices. And in a $K_4$ all edges are special edges. 

Suppose that $v\in \mathcal{U}_{\bowtie}$. By the above $v$ could be one of three types: it belongs to a $K_4$, a tip vertex of a chimney, or a base vertex of a chimney. Owing to \eqref{acl} we have the following. If $v\in Q \cong K_4$, then $\acl(v)=Q$. Otherwise $v$ lies in a chimney. If $v$ is a tip vertex, then $\acl(v)$ is the unique triangle containing $v$. If $v$ is a base vertex, then $\acl(v)$ is the unique special edge containing $v$. Moreover, it follows from \eqref{acl} that the algebraic closure is \textit{disintegrated}, that is, the algebraic closure of a set is the union of the algebraic closure of its singletons. Therefore, for a finite $A\subseteq \mathcal{U}_{\bowtie}$ we have that $\acl(A)$ is either a base of a chimney, a triangle in a chimney, a special bowtie-free graph, or a union of sets of these types.

Herwig showed in \cite[Section 6]{herwig} that a class of finite structures which has the joint embedding property and the extension property for a single partial automorphism\footnote{For any $A$ in the class, and $p\in \Part(A)$, there is an extension $B\supseteq A$ in the class with an automorphism $f\in \Aut(B)$ extending $p$.} must have the amalgamation property. Since the class of all finite bowtie-free graphs does not have the amalgamation property (see above), the extension property for partial automorphisms\footnote{See \cite[Section 5.3]{macphersonsurvey}.} (EPPA) fails. Nevertheless, when we restrict ourselves to the class of special finite bowtie-free graphs EPPA is satisfied as shown by Evans-Hubi\v{c}ka-Ne\v{s}et\v{r}il \cite[Corollary 5.5]{evanshubickanesetril}.  In the proposition below we apply an argument of Ivanov \cite[Theorem 3.1]{ivanov} to obtain a special case of \cite[Corollary 5.5]{evanshubickanesetril}; we call such argument the `necklace argument'.

\begin{proposition}\label{necklaceargument}
Suppose that $G\in \mathcal{C}_{\bowtie}^{\ast}$ is a finite special bowtie-free graph, and let $p:U\to V$ be an isomorphism between special induced subgraphs of $G$. Then there is $K\in \mathcal{C}_{\bowtie}^{\ast}$ such that $G \subseteq K$ and $p$ extends to an automorphism of $K$.
\end{proposition}

\begin{proof}
By Proposition \ref{FAP}, $\mathcal{C}_{\bowtie}^{\ast}$ has the free amalgamation property. The idea of constructing the desired graph $K$ is to form a `necklace' whose beads are isomorphic copies of $G$, and in which the range of $p$ in one bead is amalgamated with the domain of $p$ in the consecutive bead as in \cite[Theorem 3.1]{ivanov}. Start with the triple $G_0:=G$, $U_0:=U$, and  $p_0:=p$. Let $G_1$, $U_1$, and $p_1$ be new isomorphic copies of $G_0$, $U_0$, and $p_0$, respectively. Here we mean that $G_1$ is a graph isomorphic to $G_0$ via an isomorphism $\delta:G_0\to G_1$ and the set of vertices of $G_1$ is disjoint from the set of vertices of $G_0$. Moreover, $\dom(p_1)=U_1=\delta(U_0)$ and $p_1(\delta(x))=\delta(p_0(x))$ for every $x\in U_0$.   Take the free amalgam $G_0 \cup G_1 \in \mathcal{C}_{\bowtie}^{\ast}$ of $G_0$ and $G_1$ identifying $p_0(U_0)$ with $U_1$. One can check that in $G_0 \cup G_1$, the maps $p_0$ and $p_1$ agree on $U_0\cap U_1$. To see this, suppose that $v_0\in \dom(p_0)\cap \range (p_0)$. Then there are $u_0\in \dom(p_0)$ and $w_0\in \range(p_0)$ such that $p_0(u_0)=v_0$ and $p_0(v_0)=w_0$. Let $u_1, v_1, w_1$ in $G_1$ be the copies of $u_0,v_0,w_0$, respectively. So $p_1(u_1)=v_1$ and $p_1(v_1)=w_1$. In the amalgam $G_0\cup G_1$, the vertex $v_0$ is identified with $u_1$ and their corresponding images, $w_0$ and $v_1$, are identified as well. So $p_0$ and $p_1$ agree on $v_0$ in $G_0 \cup G_1$. Now using the isomorphism between $G_0$ and $G_1$ we can extend the map $p_0 \cup p_1$ to a map $g_1:G_0 \to G_1$ in $\Part(G_0 \cup G_1)$. 

We now describe how to continue the construction. Suppose that the graph $G_0 \cup \cdots \cup G_{m-1}$ together with the partial isomorphism $g_{m-1}:G_0 \cup \cdots \cup G_{m-2} \to G_1 \cup \cdots \cup G_{m-1}$ extending $p_0\cup \cdots \cup p_{m-1}$  have been constructed. Let $G_m$, $U_m$, and $p_m$ be new isomorphic copies of $G_{m-1}$, $U_{m-1}$, and $p_{m-1}$, respectively. Form the free amalgam $G_0 \cup \cdots \cup G_{m-1} \cup G_m$ in $\mathcal{C}_{\bowtie}^{\ast}$ of $G_0 \cup \cdots \cup G_{m-1}$ and $G_m$ identifying $p_{m-1}(U_{m-1})$ with $U_m$. Next, using the isomorphism between $G_{m-1}$ and $G_m$, extend the map $g_{m-1}$ to a map $g_m:G_0 \cup \cdots \cup G_{m-1} \to G_1 \cup \cdots \cup G_m$ that belongs to $\Part(G_0 \cup \cdots \cup G_{m-1} \cup G_m)$.

Choose $n$ such that $n$ is a common multiple of all the lengths of complete cycles of the isomorphism $p$, and $n$ is strictly greater than the length of any partial cycle of $p$. Let $\bar{G}=G_0 \cup \cdots \cup G_n$ and let $g:=g_n:G_0 \cup \cdots \cup G_{n-1} \to G_1 \cup \cdots \cup G_n$ in $\Part(\bar{G})$ extending $p_0 \cup \cdots \cup p_n$.  At this point, half of the necklace has been constructed. By our choice of $n$ we have the following:
\begin{enumerate}[label=(\roman*), itemsep=0pt, topsep=0pt]
    \item $G_0$ is contained in the domain of $g^n$ ($n^{th}$ power), and
    \item $G_0 \cap G_n = G_1\cap G_n = \{a\in U_0 : g^k(a)=a \text{ for some } k>0\}$, and
    \item for any $a\in G_0 \cap G_n$ we have that $g^n(a)=a$.
\end{enumerate}  
Item $(\mathrm{ii})$ says that $G_0 \cap G_n$ contains exactly the points which are in complete cycles of $p$. 

\textbf{Claim.} The induced subgraph $G_0 \cup G_n \subseteq \bar{G}$ is special. 

\textbf{Proof of the claim.} By construction of $\bar{G}$, we have that  $G_0 \cup G_n$ is the free amalgam of $G_0$ and $G_n$ over their intersection $G_0\cap G_n$. By Proposition \ref{FAP} and since both $G_0, G_n \in \mathcal{C}_{\bowtie}^{\ast}$, in order to show that $G_0 \cup G_n$ is special it is enough to show that $G_0\cap G_n$ is special. By item $(\mathrm{ii})$ we have that $v\in G_0 \cap G_n$ if and only if $v$ belongs to a complete cycle of $p$. Fix some $v\in G_0 \cap G_n$, then there is a complete $k$-cycle, say $(v=v_0, v_1, v_2, \ldots, v_{k-1})$ for some $k<\omega$ such that $v_i=p^i(v)$ for each $0\leq i < k$, $v=p^k(v)$, and $v_i$'s are distinct. Notice that $\{v_0, v_1, \ldots, v_{k-1}\}\subseteq U =\dom(p)$. As $v_0\in U$ and $U\in \mathcal{C}_{\bowtie}^{\ast}$, it follows that $v_0$ belongs to an induced subgraph $Q_0$ entirely contained in $U$ where $Q_0$ is either a $K_4$ or a chimney. Since $\range(p)=V$ and $V \in \mathcal{C}_{\bowtie}^{\ast}$ as well, there are (not necessarily distinct) copies $Q_1, \ldots, Q_{k-1}$ of $Q_0$ such that $v_i \in Q_i$, and each $Q_i \subseteq U$, and $p(Q_i)=Q_{i+1}$ where addition is performed modulo $k$. This means all vertices in $Q_0\cup Q_1\cup\cdots\cup Q_{k-1}$ are in complete cycles of $p$ and so they belong to $G_0\cap G_n$. So $Q_0 \subseteq G_0 \cap G_n$. So every vertex in $G_0 \cap G_n$ either belongs to a $K_4$ or a chimney which is contained in $G_0 \cap G_n$. Thus $G_0 \cap G_n$ is a special bowtie-free graph, and so is $G_0\cup G_n$, establishing the claim.  \hfill $\square$

Next we take a new isomorphic copy $\bar{H}=H_0 \cup H_1 \cup \ldots \cup H_n$ of the graph $\bar{G}$ and let $\gamma: \bar{G} \to \bar{H}$ be the induced isomorphism. Let the map  $h:H_0 \cup H_1 \cup \ldots \cup H_{n-1} \to H_1 \cup H_2 \cup \ldots \cup H_n$ be the corresponding copy of $g$. Here $\bar{H}$ is the other half of the necklace. Let $\beta:=g^n\restriction_{G_0}:G_0 \to G_n$ be the isomorphism induced by $g^n$. Using $\beta$ and the isomorphism $\gamma$, construct the free amalgam $K\in \mathcal{C}_{\bowtie}^{\ast}$ of $\bar{G}$ and $\bar{H}$ over $G_0 \cup G_n$ where $G_0$ is identified with $H_n$, and $G_n$ is identified with $H_0$ via the map $\gamma \circ (\beta\cup \beta^{-1}):G_0\cup G_n\to H_n\cup H_0$. Let $f:=g\cup h$. Items $(\mathrm{ii})$ and $(\mathrm{iii})$ guarantee that, under this identification, the restriction of $g$ to $G_0 \cup G_n$ agrees with the restriction of $h$ to $H_0 \cup H_n$. So $f$ is a well-defined map, and moreover, $f$ is a permutation of $K$. Finally, as $g\in \Part(\bar{G})$ and $h\in \Part(\bar{H})$ agree on $\dom(g) \cap \dom(h)$ in $K$, and $K$ is a free amalgam of $\bar{G}$ and $\bar{H}$, we have that  $f=g\cup h \in \Aut(K)$, and clearly $f$ extends $p$.    
\end{proof}

\begin{theorem}\label{universalbowtiefreegraphgenauto}
The universal bowtie-free graph $\mathcal{U}_{\bowtie}$ admits generic automorphisms.
\end{theorem}

\begin{proof}
We want to show that $\Aut(\mathcal{U}_{\bowtie})$ contains a comeagre conjugacy class via the Kechris-Rosendal characterisation \cite[Theorem 3.4]{kechrisrosendal}. To do so we pass to the Morleyisation $\ \widetilde{\mathcal{U}}_{\bowtie}$ of $\ \mathcal{U}_{\bowtie}$. Here $\ \widetilde{\mathcal{U}}_{\bowtie}$ is an expansion of $\ \mathcal{U}_{\bowtie}$ in the language  $\widetilde{\mathcal{L}}=\{R_\phi : \phi \text{ is an }  \mathcal{L}\text{-formula}\}$ where $\mathcal{L}$ is the language of graphs, and $R_\phi$ is a relation symbol of arity equal to the number of free variables in $\phi$. Moreover, the new relation symbols are interpreted as: $\ \widetilde{\mathcal{U}}_{\bowtie} \models R_\phi(\bar{a})$ if and only if $\ \mathcal{U}_{\bowtie} \models \phi(\bar{a})$ for all $\bar{a}\in \mathcal{U}_{\bowtie}$.  Thus, by \cite[Proposition 3.1.6]{macphersonsurvey} we have that $\ \widetilde{\mathcal{U}}_{\bowtie}$ is a homogeneous $\widetilde{\mathcal{L}}$-structure. 

We now show that the class of $1$-systems over the amalgamation class $\Age(\widetilde{\mathcal{U}}_{\bowtie})$ has the weak amalgamation property. So let $A\in \Age(\widetilde{\mathcal{U}}_{\bowtie})$ and $(p:U\to V) \in \Part(A)$. We may assume that $A\subseteq \widetilde{\mathcal{U}}_{\bowtie}$. 
By homogeneity of $\ \widetilde{\mathcal{U}}_{\bowtie}$, the partial automorphism $p$ extends to some $f\in\Aut(\widetilde{\mathcal{U}}_{\bowtie})$. 
Let $\bar{A}=\acl_{\widetilde{\mathcal{U}}_{\bowtie}}(A)$, $\bar{U}=\acl_{\widetilde{\mathcal{U}}_{\bowtie}}(U)$, and $\bar{V}=\acl_{\widetilde{\mathcal{U}}_{\bowtie}}(V)$. Note that $\bar{U}, \bar{V} \subseteq \bar{A}$. 
We may assume, after first increasing the universe of $\bar{A}$ slightly if necessary, that the reducts of $\bar{A}, \bar{U}, \bar{V}$ to $\mathcal{L}$ are special bowtie-free graphs. 
Moreover, the restriction of $f$ on $\bar{U}$ gives a partial automorphism $(\bar{p}:\bar{U} \to \bar{V}) \in \Part(\bar{A})$. By applying Proposition \ref{necklaceargument} to the graph reduct of $\bar{A}$ and $\bar{p}\in \Part(\bar{A})$, we obtain a special bowtie-free graph $K$ with $g\in \Aut(K)$ such that $\bar{A}\restriction_{\mathcal{L}}\subseteq K$ and $\bar{p}\subseteq g$. By Theorem \ref{universalbowtiefreegraph}(ii), $K$ can be chosen so that $K$ is a substructure of $\mathcal{U}_{\bowtie}$. Let $\bar{K}\in \Age(\widetilde{\mathcal{U}}_{\bowtie})$ be the expansion of $K$ to $\widetilde{\mathcal{L}}$, that is, equip $K$  with the induced structure when it is viewed as a subset of $\widetilde{\mathcal{U}}_{\bowtie}$.  

Now suppose that $\langle \bar{B}_1, h_1 \rangle$ and $\langle \bar{B}_2, h_2 \rangle$ are two $1$-systems over $\Age(\widetilde{\mathcal{U}}_{\bowtie})$ extending $\langle \bar{K}, g \rangle$. We may assume that the reducts $B_1, B_2$ of  $\bar{B}_1, \bar{B}_2$, respectively, to $\mathcal{L}$ are special bowtie-free graphs, and by Proposition \ref{necklaceargument} we may assume that $h_1 \in \Aut(B_1)$ and $h_2 \in \Aut(B_2)$. Let $C$ be the free amalgam of $B_1$ and $B_2$ over $K$, which is also a bowtie-free graph. So $C\in \Age(\mathcal{U}_{\bowtie})$ by Theorem \ref{universalbowtiefreegraph}(ii) and thus we may assume that $C$ is a substructure of $\mathcal{U}_{\bowtie}$. Let $\bar{C}\in \Age(\widetilde{\mathcal{U}}_{\bowtie})$ be the expansion of $C$ to $\widetilde{\mathcal{L}}$, that is, equip $C$  with the induced structure from $\widetilde{\mathcal{U}}_{\bowtie}$. Then the $1$-system $\langle \bar{C}, h_1 \cup h_2 \rangle$ amalgamates $\langle \bar{B}_1, h_1 \rangle$ and $\langle \bar{B}_2, h_2 \rangle$ over $\langle \bar{K}, g \rangle$, and so over $\langle A, p \rangle$. Therefore, the class of all $1$-systems over $\Age(\widetilde{\mathcal{U}}_{\bowtie})$ has the weak amalgamation property.

We now show that the class of all $1$-systems over $\Age(\widetilde{\mathcal{U}}_{\bowtie})$ has the joint embedding property. Let $\langle A, f \rangle$ and $\langle B, g \rangle$ be two $1$-systems over $\Age(\widetilde{\mathcal{U}}_{\bowtie})$. We may assume that the reducts of $A$ and $B$ to $\mathcal{L}$ are special bowtie-free graphs. Now let $C$ be the disjoint union of $A\restriction_{\mathcal{L}}$ and $B\restriction_{\mathcal{L}}$, and notice that $C$ is a special bowtie-graph.  By Theorem \ref{universalbowtiefreegraph}(ii), we may choose $C$ to be a substructure of $\mathcal{U}_{\bowtie}$. Let $\bar{C}$ be the expansion of $C$ to $\widetilde{\mathcal{L}}$. It follows that $\bar{C}\in \Age(\widetilde{\mathcal{U}}_{\bowtie})$ and, moreover, the 1-system $\langle \bar{C}, f\cup h \rangle$ embeds both $\langle A, f \rangle$ and $\langle B, g \rangle$ as desired. Therefore, by \cite[Theorem 3.4]{kechrisrosendal}, the automorphism group $\Aut(\widetilde{\mathcal{U}}_{\bowtie})=\Aut(\mathcal{U}_{\bowtie})$ contains a comeagre conjugacy class. That is, the universal bowtie-free graph $\mathcal{U}_{\bowtie}$ has generic automorphisms.      
\end{proof}

In the first paragraph of the proof above we passed to a homogeneous expansion of $\mathcal{U}_{\bowtie}$ using the Morleyisation technique. To do so we expanded the language of graphs to an infinite relational language. We conclude this article by showing that adding finitely many relation symbols is not sufficient to make the universal bowtie-free graph homogeneous. On this subject, see also the remark in \cite{hubickanesetril} on page 295.
 
\begin{definition}\label{finitehomog}\rm \cite[Definition 1.6]{covington} Let $\mathcal{L}$ be a finite relational language, and $M$ be a countably infinite $\mathcal{L}$-structure. We say that $M$ is \textit{finitely homogenisable} if there is a finite relational language $\tilde{\mathcal{L}} \supseteq \mathcal{L}$ and an $\tilde{\mathcal{L}}$-structure $\tilde{M}$ such that $\tilde{M}$ is an expansion of $M$, and $\tilde{M}$ is homogeneous, and $\Aut(M)=\Aut(\tilde{M})$.   
\end{definition}

\begin{remark}\label{subtuple}\rm 
 Let $\mathcal{L}$ be a finite relational language with maximum arity $k$, and $\bar{a}, \bar{b}$ be finite $\mathcal{L}$-structures of same size. Then if every $k$-subtuple of $\bar{a}$ is isomorphic to its corresponding $k$-subtuple of $\bar{b}$, then $\bar{a}$ is isomorphic to $\bar{b}$. 
 \end{remark}

\begin{figure}[ht]
\centering
\begin{tikzpicture}
\tikzset{Bullet/.style={circle,draw,fill=black,scale=0.4}}
\node[Bullet, label=below:{$a_1$}](a) at(0,0){};
\node[Bullet, label=above:{$b_1$}](b) at(1.5,1){};
\node[Bullet](c) at(0.5,1.8){};
\node[Bullet, label=right:{$t_1$}](d) at(0.2,2.8){};

\node[Bullet, label=below:{$a_2$}](a1) at(2,0){};
\node[Bullet, label=below:{$b_2$}](b1) at(3.5,1){};
\node[Bullet](c1) at(2.7,1.6){};
\node[Bullet, label=right:{$t_2$}](d1) at(2.5,2.8){};

\node[Bullet, label=below:{$a_3$}](a2) at(4,0){};
\node[Bullet, label=below:{$b_3$}](b2) at(5.5,1){};
\node[Bullet](c2) at(4.7,1.8){};
\node[Bullet, label=right:{$t_3$}](d2) at(4.5,3){};

\node[Bullet, label=below:{$a_4$}](a3) at(7,1.2){};
\node[Bullet, label=right:{$b_4$}](b3) at(8,2.2){};
\node[Bullet](c3) at(7.3,3){};
\node[Bullet, label=right:{$t_4$}](d3) at(7,4){};

\node[Bullet, label=right:{$\hat{a}$}](u) at(7.5,-1.3){};
\node[Bullet, label=right:{$\hat{b}$}](v) at(8.7,-.5){};
\node[Bullet](x) at(8.3,0.3){};
\node[Bullet, label=right:{$\hat{t}$}](y) at(8.3,1){};

\draw[thick] (a)--(b)--(c)--(a)--(d)--(b) {};
\draw[thick] (a1)--(b1)--(c1)--(a1)--(d1)--(b1) {};
\draw[thick] (a2)--(b2)--(c2)--(a2)--(d2)--(b2) {};
\draw[thick] (a3)--(b3)--(c3)--(a3)--(d3)--(b3) {};
\draw[thick] (u)--(v)--(x)--(u)--(y)--(v) {};
\draw (a)--(a1)--(a2)edge[bend right=10](a3) {};
\draw (b)--(b1)--(b2)--(b3) {};
\draw (a2)--(u) {};
\draw (b2)--(v) {};
\draw (a3)edge[bend right=20](a){};
\draw (a)edge[bend right=20](v){};
\draw (b3)edge[bend right=30](b){};
\draw (b)edge[bend right=20](u){};
\end{tikzpicture}

\caption{The special bowtie-free graph constructed in Lemma \ref{notfinhomog} for $k=3$.}\label{specialbowtiegraph3} 
\end{figure}

The argument below is based on an example by Cherlin-Lachlan \cite[p. 819]{cherlinlachlan}.

\begin{lemma}\label{notfinhomog}
The universal bowtie-free graph $\mathcal{U}_{\bowtie}$ is not finitely homogenisable. 
\end{lemma}

\begin{proof}
For the sake of a contradiction suppose that $\mathcal{U}_{\bowtie}$ is finitely homogenisable. Let $\mathcal{L}=\{E\}$ be the language of graphs, and $\tilde{\mathcal{L}}$ be the finite relational language of a homogeneous expansion $\tilde{\mathcal{U}}_{\bowtie}$ of $\mathcal{U}_{\bowtie}$. Let $k$ be the maximum arity of the symbols in $\tilde{\mathcal{L}}$. 
For each $1\leq i \leq k+1$ take a distinct chimney $H_i$ of height 2 with base vertices $\{a_i, b_i\}$ and some tip vertex $t_i$ of $H_i$, so $t_ia_ib_i$ is a triangle in $H_i$.  Construct a special bowtie-free graph by taking the disjoint union of the $H_i$'s and then adding the following edges:
$$\big\{\{a_i,a_{i+1}\}\mid 1\leq i \leq k\}\big\} \cup \big\{ \{b_i,b_{i+1}\}\mid 1\leq i \leq k\}\big\} \cup \big\{\{a_1,a_{k+1}\}, \{b_1,b_{k+1}\} \big\}.$$
Next, we add one more new chimney $\hat{H}$ of height 2 with base vertices $\{\hat{a}, \hat{b}\}$ and a tip vertex $\hat{t}$ together with the following edges:
$$\big\{\{a_k, \hat{a}\}, \{b_k, \hat{b}\}, \{a_1, \hat{b}\}, \{b_1, \hat{a}\} \big\}.$$
The resulting graph (see Figure \ref{specialbowtiegraph3}) is a special bowtie-free graph and so it embeds in $\mathcal{U}_{\bowtie}$. Consider the two $(k+1)$-tuples $\bar{u}=(t_1, t_2, \ldots, t_k, t_{k+1})$ and  $\bar{v}=(t_1, t_2, \ldots, t_k, \hat{t}\,)$. 
For every $I \subseteq \{1, 2, \ldots, k\}$ with $|I|=k-1$, one can see that there is a finite partial $\mathcal{L}$-isomorphism $f: \bigcup\limits_{i \in I} H_i \cup H_{k+1} \to \bigcup\limits_{i \in I} H_i \cup \hat{H}$ such that $f(t_i)=t_i$ for $i\in I$ and $f(t_{k+1})=\hat{t}$. 
As the domain and range of $f$ are special bowtie-free graphs, by Theorem \ref{universalbowtiefreegraph}(iii), there is $\tilde{f}\in \Aut(\mathcal{U}_{\bowtie})=\Aut(\tilde{\mathcal{U}}_{\bowtie})$ extending $f$. 
Thus, every $k$-subtuple of $\bar{u}$ is $\tilde{\mathcal{L}}$-isomorphic to its corresponding subtuple of $\bar{v}$. By Remark \ref{subtuple}, the tuples $\bar{u}$ and $\bar{v}$ are $\tilde{\mathcal{L}}$-isomorphic. 
By homogeneity of $\tilde{\mathcal{U}}_{\bowtie}$ there is some $h \in \Aut(\tilde{\mathcal{U}}_{\bowtie})$ such that $h(\bar{u})=\bar{v}$. Since $\tilde{\mathcal{U}}_{\bowtie}$ is bowtie-free and $h(t_i)=t_i$ for each $1\leq i\leq k$, we have either $h(a_i,b_i)=(a_i,b_i)$ for all $1\leq i\leq k$  or  $h(a_i,b_i)=(b_i,a_i)$ for all $1\leq i\leq k$. Suppose, without loss of generality, that $h$ fixes pointwise the base vertices of each $H_i$ for $1\leq i \leq k$. As $h(t_{k+1})=\hat{t}$, it must be that $h$ sends the base of $H_{k+1}$ to the base of $\hat{H}$, but each of the options $h(a_{k+1}, b_{k+1})=(\hat{a}, \hat{b})$ or $h(a_{k+1}, b_{k+1})=(\hat{b}, \hat{a})$ gives rise to a contradiction. To see this, suppose that $h(a_{k+1})=\hat{a}$ and $h(b_{k+1})=\hat{b}$. Since $h$ is an automorphism and $a_1$ and $a_{k+1}$ are adjacent, it must be that $h(a_1)$ and $h(a_{k+1})$ are adjacent too, and so we conclude that $a_1$ and $\hat{a}$ are adjacent, which is not the case. For the other case, suppose that $h(a_{k+1})=\hat{b}$ and $h(b_{k+1})=\hat{a}$. Since $a_k$ and $a_{k+1}$ are adjacent, it follows that $h(a_k)$ and $h(a_{k+1})$ are adjacent too, and so $a_k$ and $\hat{b}$ are adjacent, which is, again, not the case.        
\end{proof}

\textbf{Acknowledgements.} The author is extremely thankful to Dugald Macpherson for his support and insightful suggestions. He is also grateful to Rehana Patel and Jan Hubička for their beneficial and exciting conversations on bowtie-free graphs. Finally, the author would like to thank the referee for valuable comments and suggestions that improved the presentation of this work.

\bibliographystyle{abbrv}
\bibliography{DNSreferences}

\end{document}